\documentclass[a4paper,11pt]{article}
\usepackage[top=2cm,bottom=2cm,left=2cm,right=2cm]{geometry}

\usepackage{xspace}
\usepackage{authblk}
\usepackage{tikz}
\usepackage{caption}
\usepackage{subcaption}
\usetikzlibrary{mindmap}
\usepackage{amsmath}
\usepackage{amsthm}
\usepackage{bbold}
\usepackage{enumitem}
\usepackage{url}
\usepackage{lipsum}
\usepackage{graphicx}
\usepackage[flushleft]{threeparttable}

\newtheorem{proposition}{Proposition}

\newcommand{\tool}[1]{\vspace{.30cm}\phantomsection\centerline{\large \textbf{#1}}\vspace{.1cm}}

\newcommand{\ucube}{\emph{ucube}\xspace}
\newcommand{\ucubet}{\emph{ucube2}\xspace}
\newcommand{\ducube}{\emph{ducube}\xspace}
\newcommand{\ducubet}{\emph{ducube2}\xspace}
\newcommand{\sphere}{\emph{sphere}\xspace}
\newcommand{\spheret}{\emph{sphere2}\xspace}
\newcommand{\ctype}{\emph{ctype}\xspace}
\newcommand{\cps}{\emph{cp6}\xspace}
\newcommand{\hec}{\emph{hec}\xspace}
\newcommand{\Lt}{\emph{lambda}\xspace}
\newcommand{\Ltt}{\emph{lambda2}\xspace}
\newcommand{\rf}{\emph{r500}\xspace}
\newcommand{\tsp}{\emph{tsp7}\xspace}
\newcommand{\hr}{$H$-representation\xspace}
\newcommand{\vr}{$V$-representation\xspace}
\newcommand{\lrs}{\emph{lrs}\xspace}
\newcommand{\fel}{\emph{fel}\xspace}

\newcommand{\mplrs}{\emph{mplrs}\xspace}

\newcommand{\minrep}{\emph{minrep}\xspace}

\newcommand{\clark}{\emph{clark}\xspace}
\newcommand{\cddlib}{\emph{cddlib}\xspace}
\newcommand{\mai}{\emph{mai}\xspace}

\newcommand{\maief}{\emph{mai32ef}\xspace}

\newcommand{\mi}{\emph{mi}\xspace}
\newcommand{\cdd}{\emph{lcdd}\xspace}

\newcommand{\lrslib}{\emph{lrslib}\xspace}

\newcommand{\Pm}{P_{-i}}
\newcommand{\Am}{A_{-i}}
\newcommand{\bm}{b_{-i}}
\usepackage{todonotes}
\usepackage{amsmath}
\usepackage{amssymb}
\usepackage{latexsym}
\usepackage{graphicx}
\usepackage{color}
\usepackage{hyperref}
\usepackage{algorithm}
\usepackage{algpseudocode}

\DeclareMathOperator{\lin}{lin}
\DeclareMathOperator{\conv}{conv}
\DeclareMathOperator{\conic}{conic}

\makeatletter
\newcommand{\@abbrev}[3]{
  \def\c@a@def##1{
      \if ##1.
        \relax
      \else
        \@ifdefinable{\@nameuse{#1##1}}{\@namedef{#1##1}{#2##1}}
        \expandafter\c@a@def
      \fi
    }
  \c@a@def #3.
}
\@abbrev{bb}{\mathbb}{ABCDEFGHIJKLMNOPQRSTUVWXYZ}
\@abbrev{bf}{\mathbf}{ABCDEFGHIJKLMNOPQRSTUVWXYZabcdefghijklmnopqrstuvwxyz}
\@abbrev{bit}{\boldsymbol}{ABCDEFGHIJKLMNOPQRSTUVWXYZabcdefghijklmnopqrstuvwxyz}
\@abbrev{cal}{\mathcal}{ABCDEFGHIJKLMNOPQRSTUVWXYZ}
\@abbrev{frak}{\mathfrak}{ABCDEFGHIJKLMNOPQRSTUVWXYZabcdefghijklmnopqrstuvwxyz}
\@abbrev{rm}{\mathrm}{ABCDEFGHIJKLMNOPQRSTUVWXYZabcdefghijklmnopqrstuvwxyz}
\@abbrev{scr}{\mathscr}{ABCDEFGHIJKLMNOPQRSTUVWXYZ}
\@abbrev{sf}{\mathsf}{ABCDEFGHIJKLMNOPQRSTUVWXYZabcdefghijklmnopqrstuvwxyz}
\makeatother

\begin{document}

\title{Parallel Redundancy Removal in \lrslib with Application to Projections\thanks{Partially supported by JSPS 
Kakenhi Grants 
 20H00579,  
 20H00595,  
 20H05965,  
 22H05001  
 and 23K11043.  
}}
\author[1]{David Avis}
\author[2]{Charles Jordan}
\affil[1]{
   School of Informatics, Kyoto University, Kyoto, Japan and
  \\ School of Computer Science,
   McGill University, Montr{\'e}al, Qu{\'e}bec, Canada \\
    \texttt{avis@cs.mcgill.ca}
}
\affil[2]{
   Department of Information and Management Science, Otaru University of Commerce,
   Japan\\
   \texttt{skip@res.otaru-uc.ac.jp}
  }

\maketitle

\begin{abstract}
We describe a parallel implementation in \lrslib for removing redundant
halfspaces and finding a minimum representation for an \hr
of a convex polyhedron. By a standard transformation, the same code
works for \vr{}s. We use this approach to speed up the
redundancy removal step in Fourier-Motzkin elimination. Computational
results are given including a comparison with Clarkson's algorithm,
which is particularly fast on highly redundant inputs.

\noindent{}Keywords: redundancy removal, convex hulls, minimum representations, 
Fourier-Motzkin elimination, parallel processing
\end{abstract}

\section{Introduction}
\label{sec:intro}
In this section we give a general overview of the topics to be discussed, leaving formal definitions
for the next section.
In this paper we deal with convex polyhedra which we assume, to avoid trivialities, are
non-empty.
The computational problems of (1) removing redundancy, (2) finding a minimum
representation and (3) projecting a system of $m$ linear inequalities in
$\bbR^n$ (an \hr)
are fundamental in many areas of mathematics and science.
The first two problems are usually solved using linear programming (LP), and
the third via Fourier-Motzkin (F-M) elimination.
Linear programming is solvable in polynomial time and so are the first two
problems.  Projection is more difficult and was shown to be $NP$-hard by
Tiwary~\cite{Ti08}. 
Similar problems arise when the input is given by a set of vertices and
extreme rays (a \vr).
In this case, the first two problems are computationally equivalent to the
inequality setting whereas the projection problem is easier and can be
solved in polynomial time.

For the redundancy problem, the \emph{classic method} is to consider each
inequality in turn.  Checking redundancy of a given inequality (ie, whether
it can be removed without changing the solution set) can be done by solving
one LP and each redundant inequality is deleted once it is found.
This simple approach requires $m$ LPs to be solved, with the number of constraints
equal to
$m$ minus the number of redundancies already encountered.
Clarkson~\cite{Cl94} introduced an output-sensitive improvement.
Again $m$ LPs are solved
but the number of constraints is bounded by the number of non-redundant
inequalities, see Section~\ref{redund} for details.
Fukuda et al.~\cite{FGS15} also presented an output-sensitive approach to
redundancy removal based on the sign structure of all associated LP
dictionaries. The complexity contains an exponential term in the dimension
and it does not appear to have been implemented.

In general, certain inequalities may be satisfied as equations by the solution
set for the entire system and these are known as linearities.  
In the minimum representation problem, all linearities must be identified
and all redundancy removed. 
Redundant inequalities cannot be linearities.
For a given non-redundant inequality, a second LP can be used to determine if
it is a linearity.  Alternatively this can be done by solving a single large
LP as shown by Freund et al.~\cite{FRT}, although the practicality of that
method is unclear. 

The projection problem is to project the polyhedron into a subspace.
For the easier problem (projecting a \vr), we simply select the 
the coordinates of the input vectors that remain after the projection and
then find a minimum representation. The output polyhedron has fewer dimensions
and at most as many vertices and rays, so is smaller than the input.
For projecting an \hr, the F-M method eliminates one variable
at a time but the number of inequalities can increase by a quadratic factor
at each step.  Many of these may be redundant, so an efficient implementation
includes repeated redundancy removal.

The methods outlined above for redundancy removal are sequential, however
they seem good candidates for parallelization.
The classic method is particularly simple and is what we chose to implement.
The subtleties involved are a main topic of this paper.
The paper is organized as follows. The next section contains formal
definitions. Section~\ref{redund} explains how we parallelize redundancy
removal and minimum representation algorithms, and also contains a brief
description of Clarkson's algorithm.
Section~\ref{fel} shows how parallelization is used to speed up F-M
elimination and computational results are given in Section~\ref{results}. 
Finally we give some conclusions and directions for future research.

\section{Basic Definitions}
\label{sec:defs}
We begin by giving some basic definitions related to polyhedra and
linear programming. For more information, see
the books by Chv\'{a}tal~\cite{Chvatal}, Fukuda~\cite{Fukuda20} and
Ziegler~\cite{Ziegler}.
Given an $m \times n $ matrix $A=(a_{ij})$, an $m$ dimensional vector
$b$ and a possibly empty
subset $J \subseteq \{1,\ldots,m\}$ we let $A_J$ and $b_J$ denote
the submatrix of $A$ and subvector of $b$ with rows indexed by $J$.
We denote by $A_{-J}$ and $b_{-J}$ the submatrix and subvector
where the rows corresponding to the indices $J$ have been deleted.
In the case where $J=\{i\}$ is a singleton we write $A_i,b_i,\Am,\bm$
respectively. 

Let $L$ and $I$ be a partition of $\{1,\ldots,m\}$.
A
{\em convex polyhedron}, or simply {\em polyhedron}, $P$ is defined as:
\begin{equation}
\label{poly}
P=\{x\in\bbR^n:b_L+A_Lx=0,~b_I+A_Ix \ge 0 \}.
\end{equation}
This description of a polyhedron is known as an \emph{\hr} and
the rows indexed by $L$ are called linearities.
To avoid trivialities we will assume that all polyhedra discussed in
this paper are non-empty; this can be tested by a linear
program (LP).
We may also assume that the system of equations defined by $L$ is linearly
independent, using Gaussian elimination if necessary to delete dependencies.

Another way to describe $P$ is by a \emph{\vr}.
In this case we have finite sets of vectors $V,R,S$ in $\bbR^n$ of vertices, rays
and linearities. The fundamental Minkowski-Weyl theorem states that
for every $P$ defined by (\ref{poly})
\begin{equation}
\label{polyv}
P= \conv(V) + \conic(R) + \lin(S).
\end{equation} 
In words, every $x \in P$ can be expressed as the sum of a convex combination
of vertices, a nonnegative combination of rays and a linear combination
of linearities. The most fundamental problem in polyhedral computation
is the conversion of an \hr to a \vr and
vice versa. The former problem is often called 
the {\it vertex enumeration problem}
and the latter problem the {\it facet enumeration problem}. This
computation forms the core of \lrslib, see~\cite{AJ15b} for a
discussion of how it is solved in parallel.

For $i \in I$, we let $\Pm$ denote the polyhedron defined
by $\Am$ and $\bm$.
If $P=\Pm$ we say that row $i$ is {\it redundant}.
This is equivalent to saying that each $x \in \Pm$ satisfies
$b_i+A_i x \ge 0$. If each such $x$ actually satisfies $b_i +A_ix > 0$
we say row $i$ is {\it strongly redundant} otherwise it is {\it weakly redundant}. 
Finally if for each
$x \in P$ we have $b_i+A_ix = 0$,
we say that row $i$ is a {\it hidden linearity} and index $i$ can be moved
to the set $L$ if the rows indexed by $L$ remain linearly independent.
Otherwise row $i$ is deleted.

The \hr (\ref{poly}) of $P$ is {\it non-redundant}
if there are no redundant indices $i$. It is a {\it minimum representation}
if it is non-redundant and contains no hidden linearities.
In this case the dimension of $P$ is $n-|L|$ and $P$ is
{\it full dimensional} if $L$ is empty. The first part of
this paper describes a parallel method for removing redundancies
and computing a minimum description of a polyhedron based on linear programming.
We also describe Clarkson's algorithm~\cite{Cl94} which gives a much more
efficient
LP approach when the input polyhedron is highly redundant.
However, this method seems more challenging to parallelize.

Section~\ref{fel} concerns projections of polyhedra.
Let $A$ and $B$ partition the column indices $\{1,\ldots,n\}$.
For $x \in P$ we write $x=(x_A,x_B)$ to represent the corresponding
decomposition of $x$ into subspaces $\bbR^A$ and $\bbR^B$ that partition
$\bbR^n$. 
The {\it projection} of $P$ onto the subspace
$\bbR^A$ is given by
\begin{equation}
\label{proj}
P_A = \{ x_A \in \bbR^A: \exists x=(x_A,x_B) \in P \}.
\end{equation}
We will show how the parallel redundancy method described in Section~\ref{redund}
can be used to speed up the operation of the F-M
method of computing projections. 

\section{Parallel Redundancy Removal and Finding a Minimum Representation}
\label{redund}

Assume we are given an \hr(\ref{poly}) of a non-empty polyhedron $P$  
where $L$ defines a linearly independent set of equations. Choose $i \in I$ and
consider the two LPs:
\begin{equation}
\label{lpmin}
z_{\min} = \min b_i +  A_i x\phantom{a_{a}} \mbox{s.t.} ~~~
\bm + \Am x \ge 0
\end{equation}
\begin{equation}
\label{lpmax}
z_{\max} = \max b_i + A_i x\phantom{i_{i}} \mbox{s.t.} ~~~
\bm + \Am x \ge 0.
\end{equation}
The status of the $i$-th inequality is determined by the following
well known proposition based on the definitions. For completeness we give a short
proof.
\begin{proposition}
\label{basic}
The inequality $b_i+A_ix \ge 0$ is
a linearity if $z_{\max}=0$ otherwise it is
\begin{enumerate}[label=(\alph*)]
\item
weakly redundant if $z_{\min} = 0$
\item
strongly redundant if $z_{\min} > 0 $
\item
non-redundant if $z_{\min} < 0$ or unbounded 
\end{enumerate}
\end{proposition}
\begin{proof}
In the LP dictionary (see Chv\'{a}tal~\cite{Chvatal}) the $i$-th inequality
is represented using the non-negative slack variable $x_{n+i}$
as
\begin{equation}
x_{n+i} = b_i+A_ix.
\end{equation}
LP (\ref{lpmax}) seeks to find a feasible point in $P_{-i}$ that satisfies the
inequality strictly. If $z_{\max}=x_{n+i} = 0$ this is not possible so
the inequality is in fact a linearity. Otherwise LP (\ref{lpmin}) seeks
to find a point in $P_{-i}$ that violates the constraint.
If $z_{\min}=x_{n+i} \ge 0$ then the inequality cannot be violated
so it is redundant, and if $z_{\min} > 0$ it is strongly redundant.
Finally if  $z_{\min} < 0$ then there is some feasible point in $P_{-i}$ that violates the constraint
and hence it is non-redundant.
\end{proof}
It might seem that the proposition leads immediately to a parallel algorithm for redundancy removal:
check and classify each row index independently. However this fails due to the possibility
of duplicated rows, and in the presence of linearities these may be hard to discover.
For example, consider the system:
\begin{alignat*}{3}
3+{}&x_1&{}-2x_2 &  = 0 \\
    &x_1&        &\ge 0 \\
-6-{}&x_1&{}+4x_2&\ge 0.
\end{alignat*}
Both rows 2 and 3 considered independently are redundant since if we add twice row 1 to row 3
we obtain row 2. They are both weakly redundant: we can eliminate either one but not both.
But if each of these rows is considered by a different processor both will be marked redundant,
which is an error as
one of them must remain as non-redundant.
This problem becomes more acute when the system contains hidden linearities which
can easily mask duplication.
Nevertheless, inequalities classified as linearities, strongly redundant or non-redundant
will all be correctly classified. Only weakly redundant inequalities are problematic.

To solve this problem we recall that for full dimensional polyhedra, ie. when there are no
linearities, the \hr is unique up to multiplication of rows by positive scalars.
In this case we can reduce each row by its greatest common divisor (GCD) and then sort the
rows to reveal and remove duplication. Now each remaining inequality can be tested independently
and in parallel to see if it is redundant. Our general strategy will be to first find any hidden linearities
in $P$. Then we will use the linearities to eliminate variables until the resulting system
is full dimensional. 

As a first step, we can check whether the \hr(\ref{poly}) has any hidden linearities
by the single LP:
\begin{equation}
\label{lptest}
\max x_{n+1} ~~~\mbox{s.t.}~~~ b_L + A_Lx=0,~~ b_I + A_Ix\ge\mathbb{1}_{|I|}\,x_{n+1}
\end{equation}
where $\mathbb{1}_t$ denotes a column of $t$ ones. The LP terminates with $x_{n+1} > 0$
if and only if there is a point in $P$ that does not lie on the boundary of any inequality
and so there are no hidden linearities. If there are any hidden linearities then they
can be identified via Proposition~\ref{basic} and this can be done in parallel.
If there are no hidden linearities then LP (\ref{lpmax}) does not need to be solved
when classifying the inequality set $I$.
The complete procedure is described below for a polyhedron $P$
given as (\ref{poly}).

\tool{\textbf{Parallel algorithm for finding a minimum representation}}
\label{mm1}
\begin{enumerate}[label=(\alph*)]
\setlength\itemsep{-1pt}
\hrule\vspace{-.1cm}
\item\label{a}
Solve LP (\ref{lptest}) to determine if there are any hidden linearities.
If there are none, set $W=I$ and go to step~\ref{c}.
\item\label{b} (parallel)
For each $i \in I$ determine the status of $b_i+A_i x \ge 0$ according to Proposition \ref{basic}.
Place $i$ into the corresponding subset $S$ (strongly redundant), $W$ (weakly redundant),
$N$ (non-redundant) or otherwise add it to $L$ and remove it from $I$.
\item\label{c}
Remove any index $i \in L$ from $L$ for which $b_i + A_i x=0$ is linearly dependent.
\item\label{d}
For each remaining index $i \in L$, use equation $b_i + A_i x=0$
to remove one variable from $b_I + A_I x \ge 0$ by substitution.
\item\label{e}
Reduce each inequality by its GCD and eliminate any duplicate rows from $I$ obtaining an
index set $J$
and the reduced system  $\overline{b_J} + \overline{A_J} x \ge 0$. Note there is no linearity 
in this system as it is full dimensional.
\item\label{f} (parallel)
For each $i \in W \cap J$ determine the status of inequality $i$ by solving LP (\ref{lpmin})
for the reduced system, classifying them as in step~\ref{b}.
\end{enumerate}
\hrule
\vspace{0.3cm}

Observe that if there are no hidden linearities then only one LP needs to be solved for
each index in $I$. When there are hidden linearities, the number of LPs to solve depends on
the order of solving LPs (\ref{lpmin}) and (\ref{lpmax}) in step~\ref{b}. If we solve them in the order described,
then the second LP only needs to be solved when a weak redundant inequality is found.
This is the order used in \lrslib.
In the reverse order, the second LP needs to be solved whenever a linearity is not found.
In either case a further LP  is required for each weakly redundant inequality in step~\ref{f}.

A modified procedure requires at most 2 LPs to be solved per inequality when there are hidden
linearities. In step~\ref{b} we could just solve (\ref{lpmax}) and hence determine all linearities
in $I$.
We define $W$ to be all remaining inequalities and proceed as given. This approach will be faster
if most inequalities are weakly redundant, since these require only 2 LPs rather than 3. However,
if most inequalities are not weakly redundant or hidden linearities then it will be 
slower as most of the time only the LP (\ref{lpmin}) needs to be  solved.

The size of the
LPs to be solved can be greatly reduced in cases where most of the input is redundant
using a method introduced by Clarkson~\cite{Cl94}. He states his method in terms of identifying
the extreme points of a given set of input points in $\bbR^d$. 
The equivalent algorithm stated in terms of detecting redundant inequalities in an \hr
is given in Section~7.1 of~\cite{Fukuda20}.
Quoting from~\cite{Cl94} (emphasis ours):

\tool{\textbf{Clarkson's algorithm}~\cite{Cl94}}
\label{clark}
\hrule\vspace{-.1cm}
\begin{quote}
The algorithm here is as follows: process the points
of $S$ in turn, maintaining a set $E \subset S$ of extreme
points. Given $p \in S$, it is possible in $O(|E|) = O(A)$
time, using linear programming, to either show that
$p$ is a convex combination of points of $E$, or find a
witness vector $n$ for $p$, so that $n\cdot{}p > n\cdot{}q$ for all
$q \in E$. If the former, $p$ is not extremal and can be
disregarded for further consideration. If the latter,
although $p$ is not necessarily an extreme point of $S$,
one can easily in $O(n)$ time find the point $p' \in S$ that
maximizes $n\cdot{}p'$. Such a {\bf point is extremal}, and can
be added to $E$; note that it cannot already be in $E$.
\end{quote}
\hrule
\vspace{0.3cm}

Suppose the number of extreme points is $k$ which is much smaller than the input size $m$.
Then the LP to be solved can never have more than $k$ constraints compared with $m$ constraints
in the classic method. We note one point that is not mentioned in~\cite{Cl94}.
In the description above
it is assumed that $p'$ is unique. But there may be many points of $S$ on the hyperplane
$n\cdot{}x= n\cdot{}p'$. If these points are not in convex position
a non-extreme point of $S$ on the maximizing hyperplane
may be selected and marked as extremal in the output. To resolve these degenerate cases
a further recursive search may be needed on this hyperplane, increasing the 
worst-case computational complexity somewhat.

We will see in Section~\ref{results} that Clarkson's method is considerably faster than the 
classical method for inputs with high redundancy. It is usually somewhat
faster even on inputs with low redundancy
since the LPs it solves start out small, gradually increasing to the full set of
non-redundant constraints at the end of the run. In the classical method, all LPs contain
all constraints at the beginning and redundant constraints are deleted. 
To our knowledge
there is no publicly available parallel implementation of Clarkson's method and it looks
like an interesting challenge.

Finally an alternative, but usually much slower, method of computing a minimum representation is
via the $H$/$V$ transformation. Starting
with any \hr, a minimum representation of its \vr will be produced. This
can then be re-input to produce a minimum representation of the original \hr.
Although this is often impractical, it is a good way
to independently verify results on relatively small instances when testing codes.

Conversely, for many problems it is faster to first compute a minimum representation
before doing an $H$/$V$ conversion. This is because the minimum representation computation
is usually easier and the potential reduction in problem size and degeneracy speeds up
the $H$/$V$ conversion. However, in Section~\ref{results} we will see instances where
this is not the case.

\section{Projection by the Fourier-Motzkin Method}
\label{fel}

Projection of a polyhedron $P$ along coordinate axes to a lower dimension is an important
problem in many areas. For this problem the complexity is very different for \hr{}s and \vr{}s.
We start with the latter because it is very straightforward: simply delete the coordinates
of the vertices/rays/linearities that are to be projected out. This will normally generate
a redundant \vr and the methods of the last section can be used to remove any redundancies.

The F-M method can be used to project an \hr.
See~\cite{Fukuda20} or~\cite{Ziegler} for details. We give a sketch here to describe
how parallel redundancy removal can be used. The basic idea is to project out one
variable at a time. We start with an \hr (\ref{poly}) of $P$ and for simplicity describe how to
project out $x_n$. A minor modification allows the elimination of any arbitrary
variable. 

\tool{\textbf{Fourier-Motzkin elimination of $\mathbf{x_n}$}}
\label{fm}
\begin{enumerate}[label=(\alph*)]
\setlength\itemsep{-1pt}
\hrule\vspace{-.1cm}
\item\label{fela}
If there is an $i \in L$ with coefficient $a_{in} \neq 0$,
use the equation of row $i$ to eliminate $x_n$ getting a new \hr. Go to
step~\ref{fele}.
\item\label{felb} Define index sets 
\begin{equation*}
R =\{i \in I : a_{in} > 0 \} ~~~
S =\{i \in I : a_{in} < 0 \} ~~~
Z =\{i \in I : a_{in} = 0 \}
\end{equation*}
Since $b_L +A_L x = 0$ and  $b_Z +A_Z x \ge 0$
do not contain $x_n$ they remain unchanged after projecting out $x_n$.
\item\label{felc} For each $r \in R$ and $s \in S$ combine the inequalities
\begin{equation}
\label{oldin}
 -b_r - \sum_{j=1}^{n-1} a_{rj} x_j \le a_{rn} x_n~~~~~ 
-a_{sn} x_n \le b_s + \sum_{j=1}^{n-1} a_{sj} x_j
\end{equation}
obtaining 
\begin{equation}
\label{newin}
 \frac{-b_r - \sum_{j=1}^{n-1} a_{rj} x_j}{a_{rn}} \le
x_n \le
 \frac{b_s + \sum_{j=1}^{n-1} a_{sj} x_j}{-a_{sn}}.
\end{equation}
Deleting $x_n$ we get a new inequality in the remaining variables.
The new \hr has $|L|$ equations and $|Z|+|R||S|$ inequalities.
\item\label{fele}
Compute a minimum representation of the new \hr.
\end{enumerate}
\hrule
\vspace{0.3cm}
The correctness of this procedure is not difficult to establish, see either of the two
earlier references for details. By repeating the procedure a projection
onto any subset of coordinate axes can be found. 

It is clear that virtually all the computational
time will be taken in step~\ref{fele} since in the worst case there may be roughly $n^2/4$ inequalities
in the system. Various methods have been proposed to do this computation (see e.g.~\cite{Fukuda20})
but in \lrslib we use the parallel algorithm described in the previous section for finding
a minimum representation of the new \hr. A key observation is that checking for hidden linearities
only needs to be done initially for the input polyhedron $P$.

\begin{proposition}
If the input  \hr of $P$ for Fourier-Motzkin elimination is a minimum representation then the \hr produced
in either step~\ref{fela} or~\ref{felc} of the procedure will not contain hidden linearities.
\end{proposition}
\begin{proof}
By assumption, $P$ has dimension $n-|L|$. If step~\ref{fela} is executed 
one equation is
eliminated from $L$ and the number of variables becomes $n-1$. 
The dimension is unchanged and there can be no hidden linearities introduced.

If step~\ref{felc} is executed, suppose (\ref{newin}) becomes a linearity for
a certain pair $r, s$. The inequalities become equations
and we can equate coefficients obtaining $b_r/a_{rn} = b_s/a_{sn}$ and
$a_{rj}/a_{rn}=a_{sj}/a_{sn},~ j=1,\ldots,n$. 
Therefore (\ref{oldin}) defines a hidden linearity, a contradiction.
\end{proof}

As pointed out in the introduction,
it is generally much easier to project a \vr than to project
an \hr. We can make use of this fact to produce a projection of
an \hr without using F-M elimination. Let $P_H$
be the \hr of a polyhedron $P$ as in (\ref{poly}). Suppose a projection
map $\pi$ projects $P$ onto $Q$. F-M elimination directly
computes $Q_H = \pi(P_H)$. However, as shown in Figure~\ref{gold}
one can first convert $P_H$ into its \vr $P_V$, compute $Q_V=\pi(P_V)$
and finally compute $Q_H$ from $Q_V$. The first and third operations
are $H$/$V$ transformations. The success of this method depends on
doing these more efficiently than the F-M elimination
computation. We will see this in Section \ref{results}.

\begin{figure}[hb]
    \centering
    \begin{tikzpicture}
     \node (PH) at (0,2) {$P_H$};
     \node (PV) at (2,2) {$P_V$};
     \node (QH) at (0,0) {$Q_H$};
     \node (QV) at (2,0) {$Q_V$};
     \draw [thick, ->] (PH.12) -- (PV.168);
     \draw [thick, <-] (PH.-12) -- (PV.192);
     \draw [thick, ->] (QH.12) -- (QV.168);
     \draw [thick, <-] (QH.-12) -- (QV.192);
     \draw [thick, ->] (PH) -- (QH) node[midway, left] {$\pi$};
     \draw [thick, ->] (PV) -- (QV) node [midway, right] {$\pi$};
   \end{tikzpicture}
    \caption{Golden square}
    \label{gold}
\end{figure}
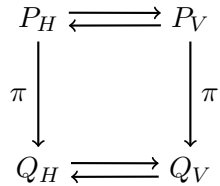

\section{Computational Results}
\label{results}
In this section we give computational results using two parallel clusters of
computers at Kyoto University.
For most results we used the \mi cluster of three similar machines containing
Ryzen Threadripper CPUs with a total of
160 cores and average clock speed of 2.8GHz. 
We made timings for 8 cores (typical laptop), 32 cores
(high performance desktop) and 160 cores (small cluster).
Some results are given using the \mai cluster with AMD Opteron
CPUs with somewhat slower 2.3GHz clock speed.
Our implementation is included in 
\lrslib v.7.3\footnote{\url{https://cgm.cs.mcgill.ca/~avis/C/lrs.html}}.
All programs used do computations in exact arithmetic.

\subsection{Redundancy Removal and Minimum Representation}
In this section we present some computational results to illustrate the speedup obtained by
using parallel processing for redundancy removal and computing a minimum representation.
The single processor version is executed by \lrs with options \texttt{testlin} and \texttt{redund},
aliased as \minrep,
and the parallel version is executed by \mplrs with the \texttt{minrep} option.
Our intention is not to do a comparison with other methods.  However,
we include results using Clarkson's algorithm \clark, as implemented in 
\cddlib v.0.94m\footnote{\url{https://github.com/cddlib/cddlib}} by Komei Fukuda,
to show the remarkable speedups it achieves for highly redundant problems.
The results are shown in Table \ref{tab:minrep} and the problems are described in
the Appendix. They range from problems with no redundancy, at the top of the table,
to problems for which almost all input is redundant, at the bottom.
As expected \clark gives best performance for the highly redundant problems.
Parallel processing gives good speedups for the classical method.

\begin{table}[htb]
\centering
\scalebox{0.9}{
\begin{tabular}{|c|c|c|c|c|c||c||c|c|c|c|}
\hline
Name    &H/V& $m_{in}$ & $d_{in}$ & $m_{out}$& redundancy& \clark & \minrep & \multicolumn{3}{|c|}{\mplrs}\\
        &   &          &          &          &  \%       &        &         & 8 cores & 32 cores & 160 cores\\
\hline
\sphere & V &  20001   &     3    &  20000   &    0.01   &  1899  &   4833  &    830  &    197   &   51  \\
\hline
\rf     & V &   500    &     100  &    500   &     0     &  6747  &  15067  &   2682  &    672   & 203   \\
\hline
\Lt     & V &   2001   &      63  &   2000   &     0.1   &  16270 &  27042  &   6137  &    2018  & 668   \\
\hline
\tsp    & H &   3447   &      21  &   3444   &     0.1   &   165  &    203  &    18   &     4    &   3   \\ 
\hline
\ucube  & H &  40000   &      6   &   3551   &     91    &   729  &   8515  &   2085  &    645   &  289  \\
\hline
\ctype  & V &   9075   &      35  &     36   &     99    &   194  &   9518  &   5656  &    721   &  203  \\
\hline
\ducube & H &  40000   &      6   &    261   &     99    &    83  &   2814  &    636  &    296   &  145  \\ 
\hline
\end{tabular}
}
\caption{Redundancy removal (time in seconds, \mi cluster) }
\label{tab:minrep}
\end{table}

\subsection{Fourier-Motzkin Elimination}
For these experiments, the input for each problem is an \hr 
and we do one round of F-M elimination
eliminating the last column.
As explained in Section~\ref{fel}, almost all of the work consists of
redundancy elimination in step~\ref{fele} of the procedure. We extracted the
inequality system created in step~\ref{felc} so that we could test 
parallelization and Clarkson's
algorithm on problems of this sort. 
The problems tested are basically the same
as before, with a few exceptions, and we use the non-redundant
version. The non-redundant description of
\ctype is a 36-dimensional simplex so projection is trivial.
For \sphere we first computed an \hr~to use as an input file.
Since the result is too big, we use the first 500 rows of the \hr
renaming the result \spheret.
Similarly, \Ltt is part of the \hr~corresponding to \Lt.
\ucubet is the non-redundant \hr of \ucube. 
We add two additional combinatorial polytopes.
The results are given in Table \ref{tab:fel}.

\begin{table}[hbt]
\centering
\begin{threeparttable}
\scalebox{0.9}{
\begin{tabular}{|c|c|c||c|c|c||c||c|c|c|c|}
\hline
Name &$m_{in}$& $d_{in}$ & $m_{FM}$ & $m_{out}$ &redundancy& \clark & \fel & \multicolumn{3}{|c|}{\mplrs} \\
 (H-reps)&    &          &          &           &   \%     &        &      &8 cores  &32 cores & 160 cores\\
\hline
\Ltt & 1080   &   63     &    4560  &    4320   &     5    & $>$300000$^{\ddagger}$   & 6981 &   1327  &   521   &   236    \\
\hline
\ducubet& 261 &    6     &   16897  &    1686   &    90    &   202  & 1310 &    329  &    87   &    29    \\
\hline
\hec    & 755 &   30     &   20029  &     949   &    95    &   455  &  237 &     76  &    53   &    29    \\
\hline
\cps    & 368 &   15     &   18592  &     224   &    99    &    59  &  273 &     55  &    15   &     9    \\
\hline
\ucubet &3551 &    6     & 3134438  &   17947   &    99    &732814$^\dagger$&-&   -  &     -   &     -    \\
\hline
\spheret& 500 &    3     &   62436  &      61   &   100    &    89  &  935 &    421  &   148   &   149    \\
\hline
\end{tabular}

}  
\begin{tablenotes}
        \item{ $\dagger$ \mai, also see Table \ref{tab:ucube}}
        \item{ $\ddagger$ suspected bug}
\end{tablenotes}

\caption{One round of F-M elimination (time in seconds, \mi cluster) }
\label{tab:fel}
\end{threeparttable}
\end{table}

In Table~\ref{tab:fel}, $m_{FM}$ is the number of new inequalities produced in step~\ref{felc}
of F-M elimination and $m_{out}$ the number of those
remaining after redundancy removal. Redundancy increases as we go down
the table and so does the efficiency of \clark. Parallel processing again
gives substantial speedups up to 32 cores but with limited improvement after that.

The problem \ucubet demonstrates the use of the golden square from Figure~\ref{gold}.
An immediate application of F-M 
generates 396,193,328 inequalities for redundancy removal,
a formidable computation.
Starting with the
non-redundant \hr of \ucubet F-M generates 3,134,438 inequalities
for redundancy removal which is still a very challenging computation. Using
\clark on \maief this took over eight days even though it is 99\% 
redundant. This direct approach to the problem is out of reach for \minrep/\mplrs,however
we can solve the problem via vertex enumeration. 
Doing so we obtain
only 303,965 vertices which we can project to lower dimension and then 
convert to an \hr. The results are given in Table~\ref{tab:ucube}.

\begin{table}[ht!]
\centering
\begin{threeparttable}
\scalebox{0.75}{
\begin{tabular}{|c||c||r|r|r|r|r||c||r|r|r|r|r|}
\hline
\ucubet  & \multicolumn{6}{|c||}{$H~\xrightarrow~V$} & \multicolumn{5}{|c|}{$V~\xrightarrow{\pi}~H_{nr}$ } \\
$m_{H}$  &  $m_{V}$& \cdd &\lrs & 8 procs & 32 procs & 160 procs  & $d_{out}$& \lrs    & 8 procs & 32 procs & 160 procs \\
\hline
 3551& 303965& 14:07$~$&:01$~$ &:01$~$ & :00$~$ &:00$~$ &5   &$>$7:00:00$~$  &$>$7:00:00$~$  &2:20:21$~$  & 1:07:55$~$\\
 &       &  &   &   &    &   & 4 & 16:43$~$   &11:21$~$  &5:29$~$&3:20$~$ \\
 &       &  &   &   &    &   & 3 &  2:25$~$ &:50$~$  &:31$~$  & :25$~$\\
 &       &  &   &   &    &   & 2 &  :02$~$ &:09$~$  &:06$~$  &:10$~$ \\
\hline
\hline
\ucubet  & \multicolumn{6}{|c||}{$V~\xrightarrow{\pi} V_{nr}$} &  \multicolumn{5}{|c|}{$V_{nr}~\xrightarrow~H_{nr}$}\\
$d_{out}$     &  $m_{V_{nr}}$& \clark &\minrep   & 8 cores & 32 cores & 160 cores &  $m_{H_{nr}}$& \lrs   & 8 cores & 32 cores & 160 cores\\
\hline
5 &121735 &4:22:35$~$&$>$7:00:00$~$ & $>$7:00:00$~$ &2:21:43$~$&23:13$~$  &17947    & $>$7:00:00 &3:03:44 &19:25$~$ &6:39$~$\\
4 &24405&2:04:23$^\dagger$&19:12:20$^\dagger$  &4:11:17$^\dagger$   &2:04:09$^\dagger$  &13:45$^\dagger$  &11817 &1:30$~$ &:25$~$ & :07$~$ &:03$~$ \\
3 &1875&:58$~$  & 1:23:28$~$ &17:55$~$ &5:24$~$  & 3:18$~$    &  1604&:00$~$  &:00$~$&:00$~$ &:00$~$ \\
2 &40& :01$~$ & 21:55$~$ &7:07$~$ &4:37$~$ &2:14$~$   &40&:00$~$  &:00$~$&:00$~$ &:00$~$ \\
\hline
\end{tabular}
}
\begin{tablenotes}
	\item{ $\dagger$ \mai cluster}
\end{tablenotes}
\caption{Projections of \ucubet by golden square (times in days:hours:minutes), \mi~cluster }
\label{tab:ucube}
\end{threeparttable}
\end{table} 

The top left part of the table shows the computation time of the
non-redundant \vr~$V$ of \ucubet
from its \hr~$H$.
We use \cdd from \cddlib to verify the results using the double description
method.
$V$ is then projected to $d=5,4,3,2$, which introduces redundancy,
and is essentially instantaneous. For each projection the top right of the 
table shows a direct computation of its \hr, which will be non-redundant
and is denoted $H_{nr}$.
For $d=5$, oly the 32-core and 160-core runs could be completed within one week.
Running times for the other dimensions are much faster, decreasing as the dimension
diminishes. Note that with F-M elimination the opposite occurs: due to
its iterative approach running times increase as the dimension diminishes.

The bottom parts of the table give the results of first removing redundancy from $V$
getting $V_{nr}$ and then using it to compute $H_{nr}$.
Most of the time is taken in the first step, shown in the bottom left part.
Again running times decrease dramatically as the dimension decreases.
Redundancy is high in the lower dimensions, and so \clark does very well there.
For $d=5$ it is interesting the running times using 32 and 160 cores are very close
to those obtained by the direct $H_{nr}$ computation.

\section{Conclusions and Future Directions}
\label{con}
We have introduced a parallelization of the classical approach to redundancy removal
which gives substantial speedups with modest hardware up to about 32 cores.
The return on increasing the number of cores is modest, possibly due to the
relatively high fixed startup computations which each processor must make. This begins
to dominate the solution time as the number of input rows to  process decreases
with the number of cores. So one future direction is to improve the scaling
to a large number of cores.

As an application we
used the codes for redundancy removal in F-M elimination, obtaining similar results.
For highly redundant problems Clarkson's algorithm, as implemented by Fukuda, gives
extremely good performance without any parallelization. As F-M elimination
can produce extremely high redundancy it is particularly well suited for this
purpose. 
An interesting challenge is to find an efficient
method to parallelize Clarkson's algorithm. 

The number of new inequalities produced by F-M elimination is highly dependent on the 
input problem, as we see in Table \ref{tab:fel}. As we saw in the case of \ucubet, it
can be considerably faster to first compute a \vr, project it, and recompute an \hr.
This is increasingly competitive for problems where it is required to project into
a relatively low dimension. A final future direction would be to produce a hybrid
code for F-M elimination that combines both methods automatically
selecting the more appropriate
method for each instance. 

\section*{Acknowledgments}
The authors would like to thank William Cook for pointing out the paper of Freund et al. and Komei Fukuda
for discussions on Clarkson's algorithm.

\bibliographystyle{plain}
\bibliography{proj}

\section*{Appendix}
\addcontentsline{toc}{section}{Appendix}
We briefly describe the test problems used in Section~\ref{results}.
\begin{itemize}
\item
\sphere is a random set of 20000 rational points on the unit sphere.
We added a redundant vertex at line 13451 of the input.
\item
\rf is a random set of 500 points in the 100-dimensional cube with coordinates between 1 and 9.
\item
\Lt derives from the Lambda polytope in quantum physics and was contributed by
Selman Ipek. For Table~\ref{tab:minrep} we added a redundant constraint at line 1393 of the input.
\item
\tsp is the seven city travelling salesman polytope. We added 3 hidden linearities.
\item
\cps is the 6 point cut polytope.
\item
\ucube, \ctype and \ducube were downloaded from Komei Fukuda's webpage:\\
\url{https://people.inf.ethz.ch/fukudak/ClarksonExp/ExperimentCtype.html} \\
(\ctype was contributed by Mathieu Dutour).
\item
\hec is the holographic cone, again from quantum physics, developed with Sergio 
Hern\'{a}ndez-Cuenca.

\end{itemize}
To get the intermediate polyhedra for input to \clark for Table \ref{tab:fel} we first
run \fel for a few seconds on the input file with the \texttt{debug} option added.
The \hr~of the intermediate polyhedron is included in the output and can be extracted.
\end{document}